\documentclass[11pt]{amsart}
\usepackage{amssymb}
\usepackage[mathscr]{eucal}
\usepackage{hyperref}
\usepackage{url}
\usepackage[all]{xy}
\usepackage{array}
\usepackage{tikz-cd}

\makeatletter
\let\@wraptoccontribs\wraptoccontribs
\makeatother

\newtheorem{thm}[equation]{Theorem}
\newtheorem{lem}[equation]{Lemma}
\newtheorem*{conjecture*}{Conjecture}

\newtheorem{cor}[equation]{Corollary}
\newtheorem{prop}[equation]{Proposition}

\theoremstyle{definition}
\newtheorem{rem}[equation]{Remark}
\newtheorem{defn}[equation]{Definition}

\numberwithin{equation}{section}

\def\F{\mathbb{F}}

\def\Fp{\F_p}

\def\A{\mathcal{A}}
\def\O{\mathcal{O}}
\def\P{\mathcal{P}}

\def\cC{\mathcal{C}}

\def\p{\mathfrak{p}}

\def\q{\mathfrak{q}}

\def\Hom{\mathrm{Hom}}

\def\Res{\mathrm{Res}}

\def\Frob{\mathrm{Fr}}
\def\image{\mathrm{image}}
\def\Sel{\mathrm{Sel}}

\def\loc{\mathrm{loc}}
\def\cond{\mathrm{cond}}

\def\ur{\mathrm{ur}}

\def\map#1{\;\xrightarrow{#1}\;}
\def\bmu{\boldsymbol{\mu}}
\def\too{\longrightarrow}

\def\dirsum#1{\underset{#1}{\textstyle\bigoplus}}

\def\Hu{H^1_\ur}

\begin{document}
\title{Increasing the $p$-Selmer rank by twisting}

\author{Minseok Kim}
\address{Department of Mathematics, Yonsei University, Seoul 03722, Republic of Korea}
\email{\href{mailto:kms727@yonsei.ac.kr}{kms727@yonsei.ac.kr}}

\markboth{Minseok Kim}{Increasing the $p$-Selmer rank by twisting}


\maketitle
\begin{abstract}
In this paper, we study the $p$-Selmer groups in the family of $p$-twists of an elliptic curve $E$ over a number field $K$. We prove that if $E/K$ is an elliptic curve over a number field $K$, and if $d$ is congruent to the dimension of the Selmer group of $E/K$ modulo $2$ and is greater than that dimension,
then there exist infinitely many characters $\chi \in \Hom(G_K, \mu_p)$ such that  
$\dim_{\F_p}(\Sel_p(E/K, \chi)) = d$ under certain conditions.

\end{abstract}
\maketitle

\section{Introduction}

Let $p$ be a prime and let $\F_p$ be the finite field with $p$-elements. For an $\F_p$-vector space $V$, we write $\dim_p(V)$ for the dimension of $V$ over $\F_p$.
For a field $k$, denote by $G(F/k)$ the Galois group of field extension $F/k$ and $G_k$ by the absolute Galois group of $k$.

Mazur and Rubin \cite[Proposition~5.2]{H10} proved that if \( G(K(E[2])/K) \cong S_3 \) or \( A_3 \), and if \( \dim_2(\Sel_2(E/K)) \ge 2 \), then $E$ has a quadratic twist \( E^\chi \) such that  
\[
\dim_2(\Sel_2(E^\chi/K)) = \dim_2(\Sel_2(E/K)) - 2.
\]  
Later, in \cite[Theorem~1]{2sel}, Yu proved that there exist infinitely many quadratic characters \( \chi \in \Hom(G_K, \{\pm1\}) \) such that
\[
\dim_2(\Sel_2(E^\chi/K)) = \dim_2(\Sel_2(E/K)) + 2,
\]
without any assumption on the Galois group \( G(K(E[2])/K) \).

In \cite{ptwist}, Mazur, Rubin and Silverberg provided a concrete definition of a $p$-twist of an elliptic curve over a number field $K$ and investigated its properties.

Furthermore, in \cite{kmr2}, Klagsbrun, Mazur and Rubin proved that there exist infinitely many cyclic characters \( \chi \in \text{Hom}(G_K, \mu_p) \) satisfying
\[
\dim_p(\text{Sel}_p(E/K,\chi)) = \dim_p(\text{Sel}_p(E/K)) + 2,
\]
under the following assumptions:
\begin{itemize}
    \item \( E[p] \) is a simple \( G_K \)-module,
    \item $\Hom_{G_{K(\mu_p)}}(E[p],E[p])=\Fp$,
    \item \( H^1(K(E[p])/K, E[p]) = 0 \).
\end{itemize}

In this paper, we establish the same result under different assumptions.  

\begin{thm} \label{thm1.1}
Let \( E \) be an elliptic curve over a number field \( K \). Suppose that one of the following conditions holds:  
\begin{enumerate}

    \item \( E(K)[p] \neq 0 \), or  \label{A1}
    \item \( [K(E[p]) : K(\mu_p)] \nmid p \), or \label{A2} 
    \item \( K = K(\mu_p) \).  \label{A3}
\end{enumerate} Then, for every positive integer \( n \), there exist infinitely many characters \( \chi \in \operatorname{Hom}(G_K, \mu_p) \) such that  
\[
\dim_p \Sel_p(E/K, \chi) = \dim_p \Sel_p(E/K) + 2n.
\]

\end{thm}

If \( E(K)[p] \neq 0 \), we strategically choose a prime \( \mathfrak{q} \) of good reduction and a certain global character \( \chi \) that is ramified at \( \mathfrak{q} \). Considering the Selmer group \( \operatorname{Sel}_p(E/K, \chi) \) in this setting, we obtain  
\[
\operatorname{loc}_{\mathfrak{q}}(\operatorname{Sel}_p(E/K)) = 0, \quad \text{but} \quad \operatorname{loc}_{\mathfrak{q}}(\operatorname{Sel}_p(E/K, \chi)) \neq 0,
\] where $\loc_\q : H^1(K,E[p]) \to H^1(K_\q,E[p])$ is the restriction map. See Definition \ref{defnsel}.
We remark that the existence of a nontrivial $p$-torsion point of $E(K)$ is crucially used. See the proof of Theorem \ref{firstthm} for details.
By the Poitou-Tate global duality, we conclude that  
\[
\dim_p(\Sel_p(E/K,\chi)) = \dim_p(\Sel_p(E/K)) + 2.
\]

Subsequently, we prove the remaining cases in Theorem \ref{thm1.1}. 
If \eqref{A2} holds, \( E/K \) satisfies some conditions as in Lemma \ref{Alem1}, which makes it easier to construct the desired global character. If \eqref{A3} holds, $E/K$ satisfies either \eqref{A1} or \eqref{A2}.

Observe that our assumption fails when all of the following conditions hold:  
\[
K \subsetneq K(\mu_p), \quad [K(E[p]):K(\mu_p)] \mid p, \quad \text{and} \quad E(K)[p] = 0.
\]
In such a case, it is difficult to construct a global character from a certain local character. See Remark \ref{caserem} for details.

Note that the Selmer group \( \Sel_p(E/K, \chi) \) is not the usual \( p \)-Selmer group of \( E/K \) or \( E^\chi/K \), but it is the $\p$-Selmer group of the abelian variety \( E^\chi/K \) (see Definition \ref{defnsel} or \cite[Proposition 5.9]{kmr1}). However, there is a relation between \( \dim_p \Sel_p(E/K, \chi) \) and \( \dim_p \Sel_p(E^\chi/K) \), where \( \Sel_p(E^\chi/K) \) is the \( p \)-Selmer group of the abelian variety \( E^\chi/K \). This leads to the following result (Corollary \ref{corollary}).

\begin{cor}
  Suppose that one of the following conditions holds:  
\begin{enumerate}
    \item \( E(K)[p] \neq 0 \), or 
    \item \( [K(E[p]) : K(\mu_p)] \nmid p \), or 
    \item \( K = K(\mu_p) \).
\end{enumerate}  For every positive integer $n$, there exist infinitely many characters $\chi \in \Hom(G_K,\mu_p)$ satisfying $\dim_p(\Sel_p(E^\chi/K)) \ge n$.
\end{cor}

As mentioned above, Mazur and Rubin \cite{H10}, and later Yu \cite{2sel}, proved that if \( G(K(E[2])/K) \cong S_3 \) or \( A_3 \), then for each \( i = \pm 2 \), there exist infinitely many quadratic characters \( \chi \) such that
\[
\dim_2(\Sel_2(E^\chi/K)) = \dim_2(\Sel_2(E/K)) + i,
\]
provided that \( \dim_2(\Sel_2(E/K)) \ge 2 \) when \( i = -2 \).

If $E/K$ has no constant $2$-Selmer parity(See \cite[Definition 9.1]{H10}), for example, $K$ has a real embedding, for $i=-1,1$, there are infinitely many quadratic characters $\chi$ such that $$\dim_2(\Sel_2(E^\chi/K))=\dim(\Sel_2(E/K))+i.$$ 
However, if $E/K$ has constant $2$-Selmer parity, this is impossible. Moreover, in \cite{bddselmer}, Klagsbrun presents an infinite family of elliptic curves $E$ defined over $K$ such that $E(K)[2]\ne 0$ and $\dim_2\Sel(E^\chi/K)\ge r_2$ for every quadratic character $\chi\in\Hom(G_K,\mu_2)$, where $r_2$ is the number of complex embeddings in $K$. This result implies that decreasing the Selmer rank is not always possible. For these reasons, we focus on increasing the Selmer rank by 2.

Our methods begin with those of \cite{2sel} and \cite{H10}. We view all the Selmer groups $\Sel_p(E/K,\chi)$ as subspaces of $H^1(K,E[p])$ as in \cite{kmr1}. We construct $\chi$ so that the local conditions defining $\Sel_p(E/K,\chi)$ and $\Sel_p(E/K)$ agree everywhere except at one place.
We then show that our $\chi$ satisfies $$\dim_p(\Sel_p(E/K, \chi)) = \dim_p(\Sel_p(E/K)) + 2.$$ Chebotarev’s density theorem ensures that there exist infinitely many such characters $\chi$. The strategy is iteratively extended to achieve larger rank increases, specifically by $+2n$ through induction.

\section{Selmer groups}
In this section, we present the lemmas required for the proof of our main theorems. Although these lemmas are not original to this work, we have included them for the reader's convenience.
Fix a prime $p\ge 3$. Let $K$ be a number field and $v$ a place of $K$. Define $\mathcal{C}(K) := \Hom(G_K,\mu_p)$ and $\cC(K_v):=\Hom(G_{K_v},\mu_p)$.
In this case, local class field theory provides a canonical identification $\cC(K_v)=\Hom(K_v^\times,\mu_p)$
and a local character is ramified if and only if it is nontrivial on the local units $\O_{K_v}^\times.$
The subsequent definitions are from \cite[Definition 5.1 and 5.3]{kmr1}.

Let $\Sigma$ be a finite set of places of $K$ containing
all places where $E$ has bad reduction, all places dividing $p\infty$,
and sufficiently large such that
\begin{itemize}
\item
the primes in $\Sigma$ generate the ideal class group of $K$,
\item
the natural map $\O_{K,\Sigma}^\times/(\O_{K,\Sigma}^\times)^p \to 
   \prod_{v\in\Sigma} K_v^\times/(K_v^\times)^p$ is injective.
\end{itemize}

\begin{rem}
 The set $\Sigma$ can always be enlarged to satisfy the above conditions, as shown in \cite[Lemma 6.1]{kmr1}.
\end{rem}
\begin{defn} \label{defnsel}
    Let $\chi \in \mathcal{C}(K)$ (or $\mathcal{C}(K_v)$) be nontrivial. Let $L$ denote the cyclic extension of $K$ (resp., $K_v$) corresponding to $\chi$. Define $$E^\chi := \ker(\Res^L_K(E) \rightarrow E)$$ where $\Res^L_K(E)$ denotes the Weil restriction of scalars of $E$ from $L$ to $K$.

    Let $\O$ denote the ring of integers of the cyclotomic field of $p$-th roots of unity, and let $\p$ denote the unique prime of $\O$ lying above $p$. Then there exists a canonical $G_K$-isomorphism $E^\chi[\p] \cong E[p]$. (See \cite[Lemma 5.2]{kmr1})

    For a place $v$ of $K$, let $$\loc_v : H^1(K,E[p]) \too H^1(K_v,E[p])$$ denote the restriction map of group cohomology and if $c\in H^1(K,E[p]),$ denote $c_v:=\loc_v(c)$.

    For a place $v$ of $K$, let $\chi$ denote an element of $\cC(K_v)$. Define $$\gamma_v(\chi):=\image(E^\chi(K_v)/{\p E^\chi(K_v)} \rightarrow H^1(K_v,E^\chi[\p])\cong H^1(K_v,E[p])).$$ 

    For a non-archimedean place $v$ with residue characteristic different from $p$, if $E$ has a good reduction at $v$, define $$\Hu(K_v,E[p]) := H^1(K_v^{\mathrm{ur}}/{K_v},E[p]),$$ where $K_v^{\ur}$ denotes the maximal unramified extension of $K_v$.

    For $\chi \in \cC(K)$, define $$\Sel_p(E/K,\chi):=\{c \in H^1(K,E[p]) : c_v \in \gamma_v(\chi_v) \text{ for all $v$} \},$$ where $\chi_v$ is the restriction of $\chi$ to $G_{K_v}$.

    If $\chi\in\cC(K)$ is trivial, define $\Sel_p(E/K):=\Sel_p(E/K,\chi).$

    Let $S$ be a set of primes of $K$. For $\psi=(\psi_v)_{v\in S} \in \prod_{v\in S} \cC(K_v)$, define 
    \begin{align*}
      \Sel_p(E/K,\psi) := \{c \in H^1(K,E[p]) : c_v \in \gamma_v(\psi_v)& \text{ for $v\in S$}, \\&c_v \in \gamma_v(1_v) \text{ for $v\notin S$} \}.
    \end{align*}

    Define $r_p(E):=\dim_p(\Sel_p(E/K))$ and $r_p(E,\chi):=\dim_p(\Sel_p(E/K,\chi)).$
\end{defn}

\begin{defn}
    For $1 \le i \le 2$, define \begin{align*}
    \P& :=\{\q : \q\notin\Sigma\}.\\
        \P_i& := \{\q\in\P : \mu_p \subset K_\q \text{  and } \dim_p \Hu(K_\q,E[p]) = i \}\\
            \P_0& :=\{\q : \q \notin\Sigma\cup\P_1\cup\P_2 \}.
    \end{align*}
    Observe that $\P=\P_0\cup\P_1\cup\P_2.$
\end{defn}

\begin{prop} \label{ramlocal}
    Assume that $v\nmid p\infty$, $E$ has good reduction at $v$ and $\chi_v$ is ramified.
    Then $\gamma_v(\chi_v) \cap \Hu(K_v,E[p])=0$.
    \end{prop}
\begin{proof}
    See \cite[Proposition 7.8]{ds}.
\end{proof}

\begin{prop} \label{localcond}
    Assume that $v\nmid p\infty$, $E$ has good reduction at $v$ and $\chi_v$ is unramified.
    Then \begin{itemize}
        \item $\gamma_v(\chi_v)=\Hu(K_v,E[p])$,
        \item $\dim_p \gamma_v(\chi_v) = \dim_p E[p]^{\Frob_v=1}$, where $\Frob_v$ denotes the Frobenius generator,
        \item there exists an isomorphism $\Hu(K_v,E[p]) \cong E[p]/{(\Frob_v-1)E[p]}$ given by evaluating cocycles at $\Frob_v$.
    \end{itemize}
\end{prop}

\begin{proof}
    See \cite[Lemma 7.2 and 7.3]{ds}
\end{proof}

\begin{thm}\label{mod2}
    Let $\chi\in\cC(K)$. We have $$r_p(E,\chi)-r_p(E)\equiv \sum_{v}h_v(\chi_v) \text{   mod $2$},$$ where $\chi_v$ is the restriction of $\chi$ to $G_{K_v}$ and $$h_v(\chi_v):=\dim_p(\gamma_v(1_v)/({\gamma_v(\chi_v)\cap\gamma_v(1_v)})).$$
\end{thm}

\begin{proof}
    See \cite[Theorem 4.11]{kmr1}.    
\end{proof}

\begin{defn}
    Let $E/K$ be an elliptic curve over a number field $K$ and let $\q$ be a place of $K$. The relaxed twisted $p$-Selmer group $\Sel_p(E/K,\chi)^\q$ at $\q$ and the strict twisted $p$-Selmer group $\Sel_p(E/K,\chi)_\q$ at $\q$ are defined by the following exact sequences : 
\begin{equation*}
\raisebox{19pt}{
\xymatrix@C=12pt@R=7pt{
0 \ar[r] & \Sel_p(E/K,\chi)^\q \ar[r] & H^1(K,E[p]) \ar^-{\dirsum{v\ne\q}\loc_v}[rr] 
   && \dirsum{v\ne\q}\displaystyle\frac{H^1(K_v,E[p])}{\gamma_v(\chi_v)} \\
0 \ar[r] & \Sel_p(E/K,\chi)_{\q} \ar[r] & \Sel_p(E/K,\chi) \ar^-{\loc_\q}[rr] 
   && \gamma_\q(\chi_\q).
}}
\end{equation*}
In particular, if $\chi$ is trivial, define $$\Sel_p(E/K)^\q:=\Sel_p(E/K,\chi)^\q \quad \text{and}\quad \Sel_p(E/K)_\q:=\Sel_p(E/K,\chi)_\q.$$
\end{defn}

\begin{thm}\label{dual}
Let $\q$ be a prime of $K$.
The images of the two right-hand maps in the following exact sequences are orthogonal complements of each other under the sum of the local Tate pairings.
\begin{equation}
\label{gdd}
\raisebox{19pt}{
\xymatrix@C=12pt@R=7pt{
0 \ar[r] & \Sel_p(E/K) \ar[r] & \Sel_p(E/K)^{\q} \ar^-{\loc_\q}[rr] 
   && \displaystyle\frac{H^1(K_\q,E[p])}{\gamma_\q(1_\q)} \\
0 \ar[r] & \Sel_p(E/K)_{\q} \ar[r] & \Sel_p(E/K) \ar^-{\loc_\q}[rr] 
   && \gamma_\q(1_\q).
}}
\end{equation}
In particular, $$\dim_p(\Sel_p(E/K)^{\q}) - \dim_p(\Sel_p(E/K)_{\q}) = \displaystyle\frac{1}{2}\dim_p(H^1(K_\q,E[p])).$$
\end{thm}

\begin{proof}
    See \cite[Theorem 2.3.4]{kolysys}
\end{proof}

\begin{rem} \label{dimrem}
    In Theorem \ref{dual}, observe that if $\q \in \P_i$, then $$\displaystyle\frac{1}{2}\dim_p(H^1(K_\q,E[p]))=i.$$ (See \cite[Proposition 7.2]{kmr2}.)
\end{rem}

\section{Increasing the Selmer rank}

In this section, we will divide it into three subsections. In the first, we will assume that \eqref{A1} holds; in the second, we will assume that \eqref{A2} holds; and in the third, we will assume that \eqref{A3} holds. For the rest of paper, let \( M:= K(E[p]) \) and \( \Sigma \) be as in Section~2.

\begin{defn}\label{defnN}
For an elliptic curve $E/K$, let $\bar{s}$ denote the image of $s \in \Sel_{p}(E/K)$ in the restriction map $$\Sel_{p}(E/K) \too \Sel_p(E/M) \subset \Hom(G_{M},E[p]).$$
Let $L_E$ be the fixed field of $\cap_{s \in \Sel_{p}(E/K)} \ker(\bar{s})$. Let $N_E$ be the Galois closure of $L_E K(\sqrt[p]{\O_{K,\Sigma}^\times})$ over $K$. Observe that $N_E$ is a finite $p$-power extension of $M$.
\end{defn}

The following lemma is frequently used to verify that $\q \in \P_i(E)$.

\begin{lem} \label{plem}
   Let $\q$ be a prime of $K$ such that $\q \notin \Sigma$, and let $\Frob_{\q} \in G(K(E[p])/K)$ denote a Frobenius element for some choice of prime above $\q$. Then :
    \begin{enumerate}
        \item $\q \in \P_2(E)$ if and only if $\Frob_{\q} = 1$;
        \item $\q \in \P_1(E)$ if and only if $\Frob_{\q}$ has order exactly $p$;
        \item $\q \in \P_0(E)$ if and only if $\Frob_{\q}^{p} \ne 1$.
    \end{enumerate}
\end{lem}

\begin{proof}
    See \cite[Lemma 4.3]{kmr1}
\end{proof}

\begin{rem} \label{zerorem}
   By the Chebotarev density theorem, $\P_2(E)$ has positive density. If $\q \in \P_0(E)$ and $\psi_\q\in\mathcal{C}(K_\q)$, $$\dim_p\Hu(K_\q,E[p])=0=\dim_pH^1(K_\q,E[p]).$$ Thus, $\gamma_\q(1_\q)=\gamma_\q(\psi_\q)=0$ and $\Sel_p(E/K,\psi_\q)=\Sel_p(E/K)$.
\end{rem}

The following two lemmas are used to construct global characters from local characters.

\begin{lem} \cite[Lemma 6.6]{kmr1}
\label{elem}
Suppose $G$ and $H$ are abelian groups, and $J \subset G \times H$ is a subgroup. Let $\pi_G$ and $\pi_H$ denote the projection maps from $G \times H$ to $G$ and $H$, respectively.
Let $J_0 := \ker(J \map{\pi_G} G/G^p)$.
\begin{enumerate}
\item The image of the natural map $\Hom((G \times H)/J,\bmu_p) \to \Hom(H,\bmu_p)$ is 
$\Hom(H/\pi_H(J_0),\bmu_p)$.
\item If $J/J^p \to G/G^p$ is injective, 
then $\Hom((G \times H)/J,\bmu_p) \to \Hom(H,\bmu_p)$ 
is surjective.
\end{enumerate}

\begin{proof}
    Consider the following sequence of $\F_p$-vector spaces is exact : $$0 \too \pi_H(J_0)H^p/H^p \too H/H^p \too (G\times H)/J(G\times H)^p.$$ Applying $\Hom(\cdot,\mu_p)$ completes the proof.
\end{proof}

\end{lem}

\begin{lem} \label{cftlemma}
    Let $\Sigma$ be a (finite) set of places of $K$ such that $\text{Pic}(\O_{K,\Sigma}^\times) = 0$. Then the image of the restriction map \begin{align*}
        \cC(K)=\Hom(G_K,\mu_p)=\Hom((\prod_{v\in \Sigma}&K_v^\times \times \prod_{v \notin \Sigma} \O_v^\times)/\O_{K,\Sigma}^\times, \mu_p) \too \\ &\prod_{v \in \Sigma} \Hom(K_v^\times, \mu_p) \times \prod_{v \notin \Sigma} \Hom(\O_v^\times, \mu_p)
    \end{align*}
     is the set of all $((\psi_v)_{v})$ such that $\prod_{v}\psi_v(b) = 1$ for all $b\in\O_{K,\Sigma}^\times$.
\end{lem}

\begin{proof}
Global class field theory, together with the assumption \( \mathrm{Pic}(\mathcal{O}_{K,\Sigma}) = 1 \), yields the equalities. The image follows as claimed.
\end{proof}

\begin{lem}\cite[Lemma 5.7]{kmr1}
\label{4.6}
Suppose $p > 2$, $\q\in\P_2$, and $\psi\in\cC(K_\q)$ is nontrivial. If $F$ is the cyclic extension of $K_\q$ corresponding to $\psi$, then
$$
\gamma_v(\psi) = \Hom(G(F/K_\q),E[p]) \subset \Hom(G_{K_\q},E[p]) = H^1(K_\q,E[p]).
$$
\end{lem}

\begin{proof}
    See \cite[Lemma 5.7]{kmr1}.
\end{proof}
\begin{rem} \label{canlem}
    In the proof of Lemma \ref{4.6}, the authors proved \begin{equation} \label{disp}
        E^\psi(K_\q)[p^\infty]=E^\psi[\p]\quad \text{and}\quad\dim_p(\Hom(G(F/K_\q),E[p]))=2.
    \end{equation} These results play a crucial role in increasing the Selmer rank.
\end{rem}

\subsection{Case1. $E(K)[p]\ne 0$}

\begin{thm}\label{firstthm}
   Suppose that $E(K)[p]\ne 0$. Then, for every positive integer $n$, there exist infinitely many $\chi \in \cC(K)$ satisfying $r_p(E,\chi)=r_p(E)+2n$.
\end{thm}

\begin{proof}
    Fix $\chi'\in\cC(K)$. We claim that there exist infinitely many $\chi''\in\cC(K)$ such that $$r_p(E,\chi'')=r_p(E,\chi')+2.$$ Then the result follows by induction.
    Fix $\chi'\in\cC(K)$. Let $\Sigma(\chi'):=\Sigma\cup\{\p\mid \cond(\chi')\}$ where $\Sigma$ is as in Section~2.
    Let $\bar{s}$ denote the image of $s \in \Sel_{p}(E/K,\chi')$ in the restriction map $$\Sel_{p}(E/K,\chi') \hookrightarrow H^1(K,E[p]) \too \Hom(G_{M},E[p]).$$
    Let $N'_E$ be the Galois closure of $L_E'K(\sqrt[p]{\O_{K,\Sigma(\chi')}^\times})$ over $K$ where $L_E'$ is the fixed field of $\cap_{s\in\Sel_p(E/K,\chi')}\ker(\bar{s})$. Choose a prime $\q\notin\Sigma(\chi')$ such that $\Frob_\q|_{N_E'}=1$. Then $\q\in\P_2(E)$ by Lemma \ref{plem}.
    Put $\psi \in \prod_{v \in \Sigma(\chi')} \Hom(K_v^\times, \mu_p) \times \prod_{v \notin \Sigma(\chi')} \Hom(\O_v^\times, \mu_p)$ so that 
    \begin{itemize}
        \item $\psi_v=1_v$ for $v\in \Sigma(\chi')$,
        \item $\psi_\q$ is not trivial, and
        \item $\psi_\p$ is trivial for $\p\notin\Sigma(\chi')\cup\{\q\}$.
    \end{itemize}
    Since $K(\sqrt[p]{\O_{K,\Sigma(\chi')}^\times}) \subset N_E'$ and $\Frob_\q|_{N_E'}=1$, we have $\psi_\q(\O_{K,\Sigma(\chi')}^\times)=1$. Hence, by Lemma \ref{cftlemma}, there exists $\chi \in \cC(K)$ such that \begin{itemize}
        \item $\chi_v=1_v$ for $v\in\Sigma(\chi')$,
        \item $\chi_\q$ is ramified,
        \item $\chi_\p$ is unramified for $\p\notin\Sigma(\chi')\cup\{\q\}$.
    \end{itemize}
     Observe that \begin{equation} \label{sbset}
         \Sel_p(E/K,\chi')=\Sel_p(E/K,\chi')_\q \subset \Sel_p(E/K,\chi'\chi) \subset \Sel_p(E/K,\chi')^\q.
     \end{equation}
    Consider the following two exact sequences : 
    \begin{equation}
\raisebox{19pt}{
\xymatrix@C=12pt@R=7pt{
0 \ar[r] & \Sel_p(E/K,\chi') \ar[r] & \Sel_p(E/K,\chi')^{\q} \ar^-{\loc_\q}[rr] 
   && \displaystyle\frac{H^1(K_\q,E[p])}{\gamma_\q(\chi'_\q)} \\
0 \ar[r] & \Sel_p(E/K,\chi')_{\q} \ar[r] & \Sel_p(E/K,\chi') \ar^-{\loc_\q}[rr] 
   && \gamma_\q(\chi'_\q).
}}
\end{equation}
By Theorem \ref{dual} and Remark \ref{dimrem}, since $\q\in\P_2(E),$$$\dim_p(\Sel_p(E/K,\chi')^{\q}) - \dim_p(\Sel_p(E/K,\chi')_{\q}) = \displaystyle\frac{1}{2}\dim_p(H^1(K_\q,E[p]))=2.$$
On the other hand, if $v\ne\q$, then $\gamma_v(\chi_v)=\gamma_v(\chi_v'\chi_v)$ by Proposition \ref{localcond}. Note that $\chi'_\q$ is unramified. Hence, by and Lemma \ref{4.6} and \cite[Theorem 1.4]{alc}, 
\begin{align*}
    r_p(E,\chi'\chi)-r_p(E,\chi') &\equiv \sum \dim_p(\gamma_v(\chi_v')/\gamma_v(\chi_v'\chi_v)\cap\gamma_v(\chi'_v)) &\pmod{2} \\
    &\equiv \dim_p(\gamma_\q(\chi_\q')/\gamma_\q(\chi'_\q\chi_\q)\cap\gamma(\chi_\q')) &\pmod{2}\\ 
    &\equiv \dim_p(\gamma_\q(\chi'_\q)) \equiv 0. &\pmod{2}
\end{align*} By our construction of $\q$ and by \eqref{sbset}, $\Sel_p(E,\chi')_\q=\Sel_p(E,\chi')$ and $0 \le r_p(E,\chi'\chi)-r_p(E,\chi')\le 2$. Hence either $$ r_p(E,\chi'\chi)=r_p(E,\chi')\quad \text{or}\quad r_p(E,\chi'\chi)=r_p(E,\chi')+2.$$
Let $\chi'':=\chi'\chi$. 
Let $f$ be the composition
$$f : E^{\chi''}(K) \too E^{\chi''}(K)/\p E^{\chi''}(K) \too H^1(K,E^{\chi''}[\p]).$$
Since $E[p]\cong E^{\chi''}[\p]$ as $G_K$-modules, there exists $P\in E^{\chi''}(K)[\p]$ such that $P\ne 0$.
Observe that the following diagram commutes : \begin{equation*}
    \begin{tikzcd}[column sep=small]
        E^{\chi''}(K)[\p] \arrow[r,"f"]\arrow[d, hook] & \Sel_p(E/K,{\chi''}) 
        \arrow[r, phantom, "="] & \Sel_p(E/K,{\chi''}) \arrow[r, hook] & H^1(K,E[p]) \arrow[d, "\loc_\q"] \\
        E^{\chi''}[\p] \arrow[r,"\sim"] & E^{\chi''}(K_\q)/\p E^{\chi''}(K_\q) 
        \arrow[r, hook] & H^1(K_\q,E^{\chi''}[\p]) \arrow[r,"\sim"] & H^1(K_\q,E[p])
    \end{tikzcd}
\end{equation*}
    Note that $E^{\chi''}[\p]\cong E^{\chi''}(K_\q)/\p E^{\chi''}(K_\q)$ canonically by Remark \ref{canlem}. By diagram chasing, $\loc_\q(f(P))\ne 0$ and thus $\Sel_p(E/K,\chi') \subsetneq \Sel_p(E,\chi'')$. Hence $r_p(E,\chi'')=r_p(E,\chi')+2$.

\end{proof}

\subsection{Case2. $[M:K(\mu_p)]\nmid p$}
Recall that $M=K(E[p])$.
The following definition and lemma are provided to make use of Lemma \ref{elem} and Lemma \ref{cftlemma}.

\begin{defn}
    Define \begin{align*}
        \A_1:=& \ker(K^\times/{(K^\times)^p} \rightarrow M^\times/{(M^\times)^p}),\\
\A_2:=&\ker(\O_{K,\Sigma}^\times/{(\O_{K,\Sigma}^\times)^p} \rightarrow \prod_{\q\in\P_0} \O_\q^\times/{(\O_\q^\times)^p}).
    \end{align*}
        
\end{defn}

\begin{rem}
    Since there is a natural injection $\O_{K,\Sigma}^\times/{(\O_{K,\Sigma}^\times)^p} \rightarrow  K^\times/{(K^\times)^p}$, we identify $\O_{K,\Sigma}^\times/{(\O_{K,\Sigma}^\times)^p}$ with its image in $K^\times/{(K^\times)^p}$.
By \cite[Lemma 6.2]{kmr1}, $\A_1$ is generated by an element $\Delta\in\O_{K,\Sigma}^\times$.
\end{rem}

\begin{lem} \label{Alem1}
    Assume that $[M:K(\mu_p)]\nmid p$. Then $\A_2 \subset \A_1$.
\end{lem}

\begin{proof}
    We claim that $\A_2 - \A_1$ is an empty set. Let $x\in\A_2 - \A_1$. Then $\sqrt[p]{x}\notin M^\times$. Since $M$ and $K(\mu_p,\sqrt[p]{x})$ are linearly disjoint over $K(\mu_p)$, by Lemma \ref{plem}, there exists a prime $\q$ of $K$ so that $\q\in\P_0(E)$ and $\Frob_\q(\sqrt[p]{x})=\zeta\sqrt[p]{x}$. However $x\in\A_2$ implies that $\sqrt[p]{x}\in\O_\q^\times$. This is a contradiction.
\end{proof}

\begin{thm} \label{2ndthm}
   Assume that $[M:K(\mu_p)]\nmid p$. Then, for every positive integer $n$, there exist infinitely many $\chi\in\cC(K)$ satisfying $r_p(E,\chi)=r_p(E)+2n.$
\end{thm}
\begin{proof}
    Let $M_E$ denote the Galois closure of $L_E$ over $K$. Choose a prime $\q_1$ of $K$ so that $\Frob_{\q_1}|_{M_E}=1$. Thus, $\q_1\in\P_2(E)$ by Lemma \ref{plem}.
    By \cite[Proposition 7.2]{kmr2}, there exist $(p-1)$-characters $\psi'_{\q_1}\in\cC(K_{\q_1})$ satisfying $r_p(E,\psi'_{\q_1})=r_p(E)+2$.
    Let $\Sigma(\q_1):=\Sigma\cup\{\q_1\}$.
Define \begin{itemize}
    \item $Q:=(\Sigma(\q_1)\cup\P_0)^c,$
    \item $J:=\O_{K,\Sigma(\q_1)}^\times,$
    \item $G:=\prod_{\p\in\P_0}\O_\p^\times,$
    \item $ H:=\prod_{\p\in Q} \O_\p^\times \times \prod_{v\in\Sigma(\q_1)}K_v^\times.$
\end{itemize} By Lemma \ref{elem}, the image of map $$\cC(K) \too \prod_{\p\in Q}\Hom(\O_\p^\times, \mu_p) \times \prod_{v\in \Sigma(\q_1)} \Hom(K_v^\times, \mu_p)$$ is equal to $$\{f\in\prod_{\p \in Q} \Hom(\O_\p^\times, \mu_p) \times \prod_{v \in \Sigma(\q_1)} \Hom(K_v^\times, \mu_p) : f(\A_2)=1 \}.$$ 
Let $$\psi=(\psi_v) \in \prod_{\p \in Q} \Hom(\O_\p^\times, \mu_p) \times \prod_{v \in \Sigma(\q_1)} \Hom(K_v^\times, \mu_p)$$ be such that
    \begin{itemize}
    \item $\psi_\p$ is trivial for $\p\in Q$,
        \item $\psi_v=1_v$ for $v\in\Sigma$,
        \item $\psi_{\q_1}=\psi'_{\q_1}$.
    \end{itemize}
     Observe that $\psi_{\q_1}(\Delta)=1$ by the construction of $\q_1$.
     Then $\psi(\Delta)=1$. By Lemma \ref{Alem1}, $\psi(\Delta)=\psi(\A_1)=1=\psi(\A_2).$ Hence, there exists ${\chi'} \in \cC(K)$ such that
    \begin{itemize}
    \item ${\chi'}_\p$ is unramified for $\p\in Q$,
        \item ${\chi'}_v=1_v$ for $v\in\Sigma$,
        \item ${\chi'}_{\q_1}=\psi'_{\q_1}$.
       \end{itemize} 
       Then $\Sel_p(E,{\chi'})=\Sel_p(E,\psi'_{\q_1})$. Thus, $r_p(E,{\chi'})=r_p(E)+2.$

Now, we will use induction. Let $\chi'$ be the global character(as in the above proof) such that $r_p(E,\chi')=r_p(E)+2n.$ Let $L$ be the fixed field of $\cap_{s\in\Sel_p(E,\chi')}\ker(\bar{s})$, where $\bar{s}$ is the image of $s$ in the restriction map $$\Sel_{p}(E/K,\chi) \hookrightarrow H^1(K,E[p]) \too \Hom(G_{M},E[p]).$$ Let $N$ be the Galois closure of $L$ over $K$. Choose a prime $\q_2\notin\Sigma$ so that $\q_2\nmid\cond(\chi')$ and that $\Frob_{\q_2}|_{N}=1$. Then $\q_2\in\P_2(E)$ and $\loc_{\q_2}(\Sel_p(E/K,\chi'))=0$. Then, there exist $(p-1)$-characters $\psi_{\q_2}\in\cC(K_{\q_2})$ such that $r_p(E,\chi',\psi_{\q_2})=r_p(E,\chi')+2$, where $\Sel_p(E,\chi',\psi_{\q_2})$ is defined by the following exact sequence: $$0 \rightarrow \Sel_p(E,\chi,\psi_{\q_2}) \rightarrow H^1(K,E[p]) \rightarrow {\frac{H^1(K_{\q_2},E[p])} {\gamma_{\q_2}(\chi'_{\q_2}\psi_{\q_2})}}\dirsum{v\ne\q_2}{\displaystyle\frac{H^1(K_v,E[p])} {\gamma_v(\chi'_v)}}.$$
Let $\Sigma(\chi'):=\Sigma\cup\{\p\notin\P_0 : \p\mid\cond(\chi')\}$, and let 
$\Sigma(\chi',\q_2):=\Sigma(\chi')\cup\{\q_2\}$.
Define \begin{itemize}
    \item $Q':=(\Sigma(\chi',\q_2)\cup\P_0)^c,$
    \item $J':=\O_{K,\Sigma(\chi',\q_2)}^\times,$
    \item $G':=\prod_{\p\in\P_0}\O_\p^\times,$
    \item $ H':=\prod_{\p\in Q'} \O_\p^\times \times \prod_{v\in\Sigma(\chi',\q_2)}K_v^\times.$
\end{itemize} By the above proof, there exists $\chi\in\cC(K)$ such that \begin{itemize}
    \item $\chi_\p$ is unramified for $\p\in Q'$,
        \item $\chi_v=1_v$ for $v\in\Sigma(\chi')$,
        \item $\chi_{\q_2}=\psi_{\q_2}$. 
       \end{itemize} 
Then $\Sel_p(E,\chi\chi')=\Sel_p(E,\chi',\psi_{\q_2})$ and hence $r_p(E,\chi\chi')=r_p(E,\chi')+2=r_p(E)+2n+2$.
\end{proof}

\subsection{Case 3. $K=K(\mu_p)$}

\begin{rem} \label{caserem}
If \([M:K(\mu_p)] \mid p\) and \( K \subsetneq K(\mu_p) \), there also exists a local character $\psi$ such that $\dim_p(\Sel_p(E/K,\psi))=\dim_p(\Sel_p(E/K))+2$ as in the proof of Theorem \ref{2ndthm}.
However, we encounter difficulties in extending the local character to a global character.
For this reason, we assume \( K = K(\mu_p) \). Then we have the follwoing corollary.
\end{rem}

\begin{cor}
    Suppose that $K=K(\mu_p)$. Then, for every positive integer $n$, there exist infinitely many $\chi\in\cC(K)$ satisfying $r_p(E,\chi)=r_p(E)+2n.$
\end{cor}

\begin{proof}
   If $[M:K(\mu_p)]\mid p$, then $E(K)[p]\ne 0$. Then the conclusion follows from Theorem \ref{firstthm}. If $[M:K(\mu_p)] \nmid p$, then the conclusion follows from Theorem \ref{2ndthm}.
\end{proof}

\begin{rem}
    Since $\p$ is the unique prime ideal of the ring of integers $\O$ of cyclotomic field of $p$-th roots of unity, $\p^{p-1}=(p)$. Consider the exact sequence : $$0 \rightarrow E^\chi[\p] \rightarrow E^\chi[p] \rightarrow E^\chi[\p^{p-2}] \rightarrow 0.$$ Taking the Galois cohomology yields the following exact sequence: $$0 \rightarrow \frac{E^\chi(K)[\p^{p-2}]}{\p E^\chi(K)[p]}\rightarrow H^1(K,E^\chi[\p]) \rightarrow H^1(K,E^\chi[p])[\p] \rightarrow 0.$$ It induces an exact sequence $$0 \rightarrow H \rightarrow \Sel_\p(E^\chi/K) \rightarrow \Sel_p(E^\chi/K)[\p] \rightarrow 0,$$ where $H$ is a subgroup of $\frac{E^\chi(K)[\p^{p-2}]}{\p E^\chi(K)[p]}$.
    By \cite[Proposition 5.9]{kmr1}, $\Sel_\p(E^\chi/K)=\Sel_p(E/K,\chi)$.
    Since $H$ is a subgroup of $\frac{E^\chi(K)[\p^{p-2}]}{\p E^\chi(K)[p]}$, $$\dim_p(H) \le \dim_p(E^\chi[\p^{p-2}]) \le 
    \dim_p(E^\chi[p]) \le 2(p-1)$$ and hence we have the following.
    \end{rem}

\begin{cor} \label{corollary}
    Let $E/K$ be an elliptic curve over a number field $K$. Suppose that one of the following conditions holds:  
\begin{enumerate}
    \item \( E(K)[p] \neq 0 \), or
    \item \( [K(E[p]) : K(\mu_p)] \nmid p \),
    \item \( K = K(\mu_p) \).
\end{enumerate} For any positive integer $n$, there exist infinitely many $\chi\in\cC(K)$ satisfying $\dim_p(\Sel_p(E^\chi/K)) \ge n.$ 
\end{cor}

\section*{Acknowledgments}
The author is grateful to his PhD advisor, Myungjun Yu, for the guidance and assistance provided during the course of this research. The author was supported by the National Research Foundation of Korea (NRF) grant funded by the Korea government (MSIT) (No. 2020R1C1C1A01007604).
He would also like to thank the referee for many helpful comments that greatly improved the readability of the paper.

\end{document}